\documentclass{article}
\usepackage[utf8]{inputenc}
\usepackage{amsmath,amsthm,amssymb}
\usepackage{mathrsfs}
\usepackage{fullpage}

\newtheorem{theorem}{Theorem}[section]
\newtheorem{lemma}[theorem]{Lemma}
\newtheorem{corollary}[theorem]{Corollary}
\newtheorem{proposition}[theorem]{Proposition}

\theoremstyle{definition}			                						

\newtheorem{remark}[theorem]{Remark}

\DeclareMathOperator{\diag}{diag}

\title{Polynomials that preserve nonnegative matrices of order two}
\author{
Benjamin J.~Clark\thanks{
Supported by a 2018 University of Washington Bothell Founder’s Fellow Award. Address: University of Washington Bothell, Bothell, WA 98011-8246. E-mail: {\tt thrin@uw.edu}.} 
\and 
Pietro Paparella\thanks{
†Supported by a 2021 University of Washington Bothell Scholarship, Research, and Creative Practice
(SRCP) Seed Grant. Address: University of Washington Bothell, Bothell, WA 98011-8246. E-mail: {\tt pietrop@uw.edu}.}}
\date{\today}

\begin{document}
\maketitle

\begin{abstract}
A known characterization for entire functions that preserve all nonnegative matrices of order two is shown to characterize polynomials that preserve nonnegative matrices of order two. Equivalent conditions are derived and used to prove that $\mathscr{P}_3 \subset \mathscr{P}_2$, which was previously unknown. A new characterization is given for polynomials that preserve nonnegative circulant matrices of order two. 

\vspace*{12pt}

\vspace*{12pt}
\end{abstract}

\section{Introduction}

In 1979, Loewy and London \cite{ll1978-79} posed the problem of characterizing     
\[\mathscr{P}_n := \left\{ p \in \mathbb{C}[x] \mid p(A) \geq 0, \forall A \in \mathsf{M}_n (\mathbb{R}), A \geq 0  \right\}, \]
for every positive integer $n$. In particular, and for practical purposes, necessary and sufficient conditions are sought in terms of the coefficients of the polynomials belonging to $\mathscr{P}_n$. 

The characterization of $\mathscr{P}_1$ is known as the P\'{o}lya--Szeg\"{o} theorem (see, e.g., Powers and Reznick \cite[Proposition 2]{pr2000}), which asserts that $p \in \mathscr{P}_1$ if and only if 
\[ p(x) = \left(f_1(x)^2 +f_2(x)^2\right) + x\left(g_1(x)^2 + g_2(x)^2\right). \] 

Bharali and Holtz \cite{bh2008} gave partial results for the set 
\[\mathscr{F}_n := \left\{ f~\text{entire} \mid f(A) \geq 0, \forall A \in \mathsf{M}_n (\mathbb{R}), A \geq 0  \right\} \supset \mathscr{P}_n \]
and characterized entire functions that preserve certain structured nonnegative matrices, including upper-triangular matrices and circulant matrices. In addition, they gave necessary and sufficient conditions for an entire function $f$ to belong to $\mathscr{F}_2$. Specifically, they showed that an entire function $f$ belongs to $\mathscr{F}_2$ if and only if  
\begin{equation}
\label{bhcond}
f(x + y) - f(x - y) \ge 0,~\forall x, y \ge 0,    
\end{equation}
and
\begin{equation}
\label{bhcond2}
(x + y - z)f(x - y)+(z - x + y)f(x + y) \ge 0,~\forall x \ge z \ge 0, y \ge x - z,
\end{equation}
or, equivalently, if $f$ satisfies \eqref{bhcond} and 
\begin{equation}
\label{bhcond3}
(x + y)f(x - y)+(y - x)f(x + y) \ge 0,~\forall y \ge x \ge 0.    
\end{equation}

More recently, Clark and Paparella \cite{cp2022} gave partial results for $\mathscr{P}_n$ in terms of the coefficients of the polynomials in $\mathscr{P}_n$. While it is known that $\mathscr{P}_{n+1} \subseteq \mathscr{P}_n$, $\forall n \in \mathbb{N}$, Clark and Paparella proved that $\mathscr{P}_2 \subset \mathscr{P}_1$ and conjectured that $\mathscr{P}_{n+1} \subset \mathscr{P}_n$, $\forall n \in \mathbb{N}$. 

In this work, it is shown that the characterization for $\mathscr{F}_2$ established by Bharali and Holtz also characterizes $\mathscr{P}_2$. Our demonstration, which utilizes the definition of matrix function via \emph{Jordan canonical form}, directly establishes that \eqref{bhcond} and \eqref{bhcond3} are necessary and sufficient whereas Bharali and Holtz establish \eqref{bhcond} and \eqref{bhcond2} (via the definition of matrix function via \emph{interpolating polynomial}) and proceed to show that \eqref{bhcond2} is equivalent to \eqref{bhcond3}. Equivalent conditions are derived for \eqref{bhcond} that are used to prove that $\mathscr{P}_3 \subset \mathscr{P}_2$, which was previously unknown. A new characterization is given for polynomials that preserve nonnegative circulant matrices of order two.

\section{Notation and Background}

The set of $m$-by-$n$ matrices with entries from a field $\mathbb{F}$ is denoted by $\mathsf{M}_{m\times n}(\mathbb{F})$. If $m = n$, then $\mathsf{M}_{m\times n}(\mathbb{F})$ is abbreviated to $\mathsf{M}_{n}(\mathbb{F})$. The set of all $n$-by-$1$ column vectors is identified with the set of all ordered $n$-tuples with entries in $\mathbb{F}$ and thus denoted by $\mathbb{F}^n$. 

If $A \in \mathsf{M}_n(\mathbb{F})$, then $a_{ij}$ denotes the $(i,j)$-entry of $A$. If $\mathbb{F} = \mathbb{R}$ and $a_{ij} \ge 0$ ($a_{ij} > 0$), $1 \le i,j \le n$, then $A$ is called \emph{nonnegative} (respectively, \emph{positive}) and this is denoted by $A \geq 0$ (respectively, $A > 0$). 

Unless otherwise stated,   
\[ p(x) = \sum_{k=0}^m a_k x^k \in \mathbb{C}[x],\]
where $a_m \ne 0$. If $n$ is a positive integer less than or equal to $m$, then the coefficients $a_0, a_1,\dots, a_{n-1}$ are called the \emph{first $n$ terms of $p$} and the coefficients $a_{m-n + 1}, \dots, a_{m-1}, a_m$ are called the \emph{last $n$ terms of $p$}.

\section{Basic Observations}

\begin{lemma}
    \label{diagsim}
If $D$ is a positive diagonal matrix, then $p(A) \ge 0$ if and only if $p(D^{-1} A D) \ge 0$.
\end{lemma}

\begin{proof}
The result follows immediately by observing that \( p(D^{-1} A D) = D^{-1} p(A) D\). 
\end{proof}

\begin{lemma}
    \label{permsim}
If $P$ is a permutation matrix, then $p(A) \ge 0$ if and only if $p(P^\top{A}P) \geq 0$.
\end{lemma}

\begin{proof}
Similar to the proof of Lemma \ref{diagsim}.
\end{proof}

We briefly digress to present the following result which, to the best of our knowledge, has not previously been addressed in the literature.

\begin{theorem}
\label{realcoeff}
If $p \in \mathscr{P}_1$, then $a_k \in \mathbb{R}$, $\forall k \in \{0,1,\ldots,m\}$. 
\end{theorem}

\begin{proof}
It is known that if $f$ is an analytic function defined on a \emph{self-conjugate} domain $\mathcal{D} \subseteq \mathbb{C}$ (i.e., $\mathcal{D}$ is symmetric with respect to the real-axis in the complex-plane) and $f(x) \in \mathbb{R}$, $\forall x \in \mathcal{I}:=\mathcal{D}\cap\mathbb{R}$, then $f^{(k)}(x) \in \mathbb{R}$, $\forall x \in \mathcal{I}$ (see, e.g., Paparella \cite[Lemma 4.7]{p2015}). In particular, $p^{(k)}(x) \in \mathbb{R}$, $\forall x \ge 0$. The result follows by noting that $a_k = p^{(k)}(0)/k! \in \mathbb{R}$. 
\end{proof}

\begin{corollary}
\label{containment}
If $p \in \mathscr{P}_n$, then $a_k \in \mathbb{R}$, $\forall k \in \{0,1,\ldots,m\}$. 
\end{corollary}

\begin{proof}
Since $\mathscr{P}_{n+1} \subseteq \mathscr{P}_n$, $\forall n \in \mathbb{N}$ \cite[Lemma 1]{bh2008}, it follows that $p \in \mathscr{P}_1$. The result is now immediate from Theorem \ref{realcoeff}. 
\end{proof}

\section{A Characterization of $\mathscr{P}_2$}

\begin{lemma}
\label{posmatrix}
Let $A \in \mathsf{M}_2(\mathbb{R})$ and suppose that $A > 0$. If $\sigma(A) = \{ \rho, \mu \}$, with $\rho > \vert \mu \vert$, then $A$ is similar to a matrix of the form
\[ 
\frac{1}{1 + \alpha}
    \begin{bmatrix}
        \alpha\rho + \mu    & \rho - \mu        \\
        \alpha(\rho - \mu)  & \alpha\mu + \rho
    \end{bmatrix}, 
\]
where $\alpha > 0$.
\end{lemma}

\begin{proof}
By the Perron--Frobenius theorem for positive matrices, there is a positive vector $x$ such that $Ax = \rho x$. If $D = \diag(x_1, x_2)$, then the positive matrix 
\[ B := D^{-1} A D \]
has row sums equal to $\rho$. Thus, there is an invertible matrix $\hat{S} = \begin{bmatrix}
1 & \hat{a} \\
1 & \hat{b}
\end{bmatrix}$ 
such that 
\[ 
B = \hat{S}
\begin{bmatrix}
\rho & 0 \\
0 & \mu 
\end{bmatrix} \hat{S}^{-1}.
\]
Notice that $\hat{a} \neq 0$ and $\hat{b} \neq 0$: for contradiction, if $\hat{a} = 0$, then
\[B \begin{bmatrix}
0 \\
\hat{b}
\end{bmatrix} = \mu \begin{bmatrix}
0 \\
\hat{b}
\end{bmatrix}, \] 
but 
\[ B \begin{bmatrix}
0 \\
\hat{b}
\end{bmatrix} = \hat{b} \begin{bmatrix}
b_{12} \\
b_{22}
\end{bmatrix}. \]
Thus, $b_{12} = 0$, but this is a contradiction since $B > 0$. A similar calculation demonstrates that $\hat{b} \neq 0$. 

If 
\[ 
S := \hat{S} 
\begin{bmatrix}
1 & 0 \\
0 & 1/\hat{a}
\end{bmatrix} = 
\begin{bmatrix}
1 & 1 \\
1 & a
\end{bmatrix}, \]
where $a = \hat{b} / \hat{a}$, then 
\[ B = S \begin{bmatrix}
\rho & 0 \\
0 & \mu 
\end{bmatrix} S^{-1}. \] 
Furthermore, $a < 0$ (otherwise, 
\[ B = 
\frac{1}{1-a}
\begin{bmatrix}
a\rho-\mu & \mu - \rho \\
a(\rho - \mu) & a\mu - \rho 
\end{bmatrix} \]
and $b_{12} < 0$). Thus,
\[ 
S = 
\begin{bmatrix}
1 & 1 \\ 
1 & -\alpha
\end{bmatrix},~\alpha > 0, \]
and 
\begin{equation}
    \label{altmatrix}
    B = 
    \frac{1}{1 + \alpha}
    \begin{bmatrix}
    \alpha\rho + \mu    & \rho - \mu        \\
    \alpha(\rho - \mu)  & \alpha\mu + \rho
    \end{bmatrix},
\end{equation}
as desired.
\end{proof}

To simplify the main result, we rely on the following result \cite[Lemma 4]{bh2008}.   

\begin{lemma}
\label{bhlemma}
If $p \in \mathbb{R}[x]$, then $p \in \mathscr{P}_n$ if and only if $p(A) \ge 0$ whenever $A > 0$. 
\end{lemma}

\begin{proof}
Follows from the continuity of $p$ and the fact that the set of positive matrices of order $n$ is dense in the set of all nonnegative matrices of order $n$.
\end{proof}

\begin{theorem}
[cf.~{\cite[Theorem 13]{bh2008}}]
\label{thm:p2char}
If $p \in \mathbb{R}[x]$, then $p \in \mathscr{P}_2$ if and only if 
\begin{equation}
p(\rho) \geq \vert{p(\mu)}\vert,~\forall \rho, \mu \in \mathbb{R}, \rho \ge \vert{\mu}\vert \label{spectralcond}
\end{equation}
and
\begin{equation} 
\rho p(-\mu) + \mu p(\rho) \ge 0, \label{ratiocond} 
\end{equation}
whenever $0 < \mu \le \rho$.
\end{theorem}

\begin{proof}
If $p \in \mathscr{P}_2$, then the necessity of \eqref{spectralcond} follows by noting that if 
\begin{equation}
\label{twocirc}
A := \frac{1}{2} 
\begin{bmatrix}
\rho + \mu  & \rho - \mu \\
\rho - \mu  & \rho + \mu
\end{bmatrix} \ge 0, 
\end{equation}
with $\rho \ge \vert{\mu}\vert$, then
\begin{equation}
\label{polycirc}
p(A) = \frac{1}{2} 
\begin{bmatrix}
p(\rho) + p(\mu)  & p(\rho) - p(\mu) \\
p(\rho) - p(\mu)  & p(\rho) + p(\mu)
\end{bmatrix} \ge 0.
\end{equation}

Let $\rho$ and $\mu$ be real numbers such that $0 < \mu \le \rho$. If 
\begin{equation*}
    A :=
    \begin{bmatrix}
    0       & \rho                      \\
    \mu    & \rho - \mu  
    \end{bmatrix} \ge 0
\end{equation*}
then  
\begin{align*}
A 
&= \begin{bmatrix}
\rho    & 1  \\
-\mu     & 1 
\end{bmatrix}
\begin{bmatrix}
-\mu & 0 \\
0 & \rho
\end{bmatrix}
\begin{bmatrix}
\rho    & 1  \\
-\mu     & 1 
\end{bmatrix}^{-1}  \\
&=
\frac{1}{\rho + \mu}
\begin{bmatrix}
\rho    & 1  \\
-\mu     & 1 
\end{bmatrix}
\begin{bmatrix}
-\mu & 0 \\
0 & \rho
\end{bmatrix}
\begin{bmatrix}
1       & -1     \\  
\mu    & \rho 
\end{bmatrix},
\end{align*}
and 
\begin{equation*}
    p(A) =
    \frac{1}{\rho + \mu}
    \begin{bmatrix}
    \rho p(-\mu) + \mu p(\rho)   & \rho(p(\rho) - p(-\mu))      \\
    \mu(p(\rho) - p(-\mu))       & \rho p(\rho) + \mu p(-\mu) 
    \end{bmatrix} \ge 0.
\end{equation*}
i.e., $p$ satisfies \eqref{ratiocond}.

Conversely, suppose that $p$ satisfies \eqref{spectralcond} and \eqref{ratiocond}. In view of Lemma \ref{bhlemma}, it suffices to show that $p$ maps positive matrices of order two to nonnegative matrices of order two. To this end, suppose that $A$ is a positive matrix of order two with spectrum $\{ \rho, \mu \}$. Without loss of generality, assume that $\rho > \vert\mu\vert$. 

By Lemma \ref{posmatrix}, $A$ is similar to a matrix of the form
\[ 
    B =
    \frac{1}{1 + \alpha}
    \begin{bmatrix}
        \alpha\rho + \mu    & \rho - \mu        \\
        \alpha(\rho - \mu)  & \alpha\mu + \rho
    \end{bmatrix},
\]
where $\alpha > 0$. In view of Lemmas \ref{diagsim} and \ref{permsim}, it suffices to show that $p(B) \ge 0$. By \eqref{altmatrix}, notice that
\[
p(B) 
= Sp(D)S^{-1} 
= \frac{1}{1 + \alpha}
\begin{bmatrix}
\alpha p(\rho) + p(\mu)     & p(\rho) - p(\mu)        \\
\alpha(p(\rho) - p(\mu))    & \alpha p(\mu) + p(\rho)
\end{bmatrix}.
\]
Since $p$ satisfies \eqref{spectralcond}, it follows that $p \in \mathscr{P}_1$. Thus, $p(B) \ge 0$ whenever $\mu\ge{0}$. 

If $\mu < 0$, then, since $B > 0$, it follows that $\alpha > -{\mu}/{\rho} = {\vert\mu\vert}/{\rho}$. Thus, 
\[ 
\alpha p(\rho) + p(\mu) > \frac{\vert\mu\vert}{\rho} p(\rho) + p(-\vert\mu\vert) = \frac{\vert\mu\vert p(\rho) + \rho p(-\vert\mu\vert)}{\rho} \ge 0 
\]
by \eqref{ratiocond}. The remaining entries of $p(B)$ are nonnegative by \eqref{spectralcond}.
\end{proof}

\section{Equivalent Conditions for $\mathscr{P}_2$}

\begin{proposition}
    \label{altchar}
If $p \in \mathbb{R}[x]$, then 
\begin{equation}
p(x) \geq \vert{p(y)}\vert,~\forall x, y \in \mathbb{R}, x \ge \vert{y}\vert \label{modcond}
\end{equation}
if and only if 
\begin{equation}
p' \in \mathscr{P}_1 \label{derivpone}   
\end{equation}
and 
\begin{equation}
p(x) \geq |p(-x)|,~\forall x \geq 0. \label{modcondtwo}
\end{equation}
\end{proposition}

\begin{proof}
First, note that $p \in \mathscr{P}_1$ whenever $p$ satisfies \eqref{modcond} or \eqref{modcondtwo}.

If \eqref{modcond} holds, then \eqref{modcondtwo} clearly holds. To demonstrate \eqref{derivpone}, for contradiction, let $x \ge 0$ and $h > 0$. By \eqref{modcond}, $p(x + h) \ge \vert p(x) \vert \ge p(x)$. Hence, $p(x+h) - p(x) \ge 0$. Dividing by $h$ and letting $h \to 0^+$ shows that $p' \in \mathcal{P}_1$. 

Assume that $a, b \in \mathbb{R}$, with $a \geq |b|$. By assumption, $p' \in \mathscr{P}_1$ and so $p$ is increasing on $[0,\infty)$. Using this and \eqref{modcondtwo}, we obtain \( p(a) \ge p(\vert b \vert) \ge \vert p(b) \vert \).
\end{proof}

Recall that if $f: \mathbb{C} \longrightarrow \mathbb{C}$, then 
\[ f_e(x) := \frac{f(x) + f(-x)}{2} \]
is called the \emph{even-part of $f$} and  
\[ f_o(x) := \frac{f(x) - f(-x)}{2} \]
is called the \emph{odd-part of $f$}.  

\begin{proposition}
\label{prop:evenoddcondition}
If $p: \mathbb{C} \longrightarrow \mathbb{C}$, then $p$ satisfies \eqref{modcondtwo} if and only if $p_e, p_o \in \mathscr{P}_1$.
\end{proposition}

\begin{proof}
Notice that  
\begin{align*}
    p_e, p_o \in \mathscr{P}_1 
    &\iff \frac{p(x) + p(-x)}{2} \geq 0 \text{ and } \frac{p(x) - p(-x)}{2} \geq 0,~\forall x \ge 0  \\
    &\iff p(x) + p(-x) \geq 0 \text{ and } p(x) - p(-x) \geq 0,~\forall x \ge 0                      \\
    &\iff \vert p(-x) \vert \leq p(x),~\forall x \ge 0,
\end{align*}
as desired.
\end{proof}

\begin{theorem}
Conditions \eqref{spectralcond} and \eqref{ratiocond}  are independent.
\end{theorem}

\begin{proof}
If $p(x) = x^5 - 2x^3 + 2x$, $\rho = 1$, and $\mu = .5$, then 
\[ \rho p(-\mu) + \mu p(\rho) = -0.78125 < 0, \] 
i.e., $p$ does not satisfy equation \eqref{ratiocond}.

Clearly, $p_e \in \mathscr{P}_1$ since $p_e(x) = 0$. Since 
\[ p'(x) = 5x^4 - 6x^2 + 2 = 5\left(x^2 - \frac{3}{5}\right)^2 + \frac{1}{5} \]
and 
\[ p_o(x) = p(x) = x((x^2-1)^2 + 1^2), \]
it follows that $p', p_o \in \mathscr{P}_1$. Thus, $p$ satisfies \eqref{spectralcond} by Propositions \ref{altchar} and \ref{prop:evenoddcondition}.

If $p(x) = -x$, then $p$ does not satisfy \eqref{spectralcond}. If $0 < \mu \le \rho$, then 
\[
\rho p(-\mu) + \mu p(\rho) = \rho\mu - \rho\mu = 0,
\]
i.e., $p$ satisfies \eqref{ratiocond}.
\end{proof}

\begin{theorem}
\label{thm:altchar}
If $p \in \mathbb{R}[x]$, then $p \in \mathscr{P}_2$ if and only if $p',p_e,p_o \in \mathscr{P}_1$ and $p$ satisfies \eqref{ratiocond}.
\end{theorem}

\begin{proof}
Immediate from Theorem \ref{thm:p2char} and Propositions \ref{altchar} and \ref{prop:evenoddcondition}.
\end{proof}

\begin{remark}
The preceding arguments also apply to \emph{entire functions}; as such, $f \in \mathscr{F}_2$ if and only if $f',f_e,f_o \in \mathscr{P}_1$ and $f$ satisfies \eqref{ratiocond}.   
\end{remark}

\begin{proposition}
\label{prop:evenoddratiocondition}
If $p$ is a polynomial such that $p_o$ satisfies \eqref{ratiocond} and $p_e \in \mathcal{P}_1$, then $p$ satisfies \eqref{ratiocond}.
\end{proposition}

\begin{proof}
The result follows with the observation that
\begin{align*}
\mu p(\rho)+\rho p(-\mu) 
&= \mu(p_e(\rho) + p_o(\rho)) + \rho(p_e(-\mu) + p_o(-\mu))                                             \\
&= \left(\mu p_e(\rho) + \rho p_e(\mu)\right) + \left(\mu p_o(\rho) + \rho p_o(-\mu)\right),
\end{align*}
which is nonnegative by the hypotheses.
\end{proof}

Clark and Paparella \cite[Conjecture 5.2]{cp2022} conjectured that $\mathscr{P}_{n+1} \subset \mathscr{P}_{n}$, $\forall n \in \mathbb{N}$ and showed that $\mathscr{P}_2 \subset \mathscr{P}_1$. The following result settles the conjecture when $n=2$. 

\begin{theorem}
$\mathscr{P}_3 \subset \mathscr{P}_2$. 
\end{theorem}

\begin{proof}
Consider the polynomial $p(x) = x^4 - x^2 + x + 1$. It is known that if 
\[ p(x) = \sum_{k=0}^m a_k x^k \in \mathscr{P}_n,~a_m \ne 0, \] 
and $m \ge n-1$, then $a_k \ge 0$, $\forall k \in \{0,1,\ldots,n-1\}$ (see Bharali and Holtz \cite[Proposition 2]{bh2008} or Clark and Paparella [Corollary 4.2]\cite{cp2022}). Thus, $p \notin \mathscr{P}_3$.

Since $p_o(x) = x$, it is clear that $p_o \in \mathscr{P}_1$. Notice that $p', p_e \in \mathscr{P}_1$ since 
\[ p'(x) = 4x^3 - 2x^2 + 1 = x [(2x-1)^2  + 1^2] + [(2x-1)^2 + 0^2] \]
and 
\[ p_e(x) = x^4 - x^2 + 1 = \left(x^2 - \frac{1}{2} \right)^2 + \frac{3}{4}. \] 

We also have that $p_o$ satisfies \eqref{ratiocond} since 
\[
\rho p_o(-\mu) + \mu p_o(\rho) = -\rho\mu + \mu \rho = 0.
\]
By Proposition \ref{prop:evenoddratiocondition}, $p$ satisfies \eqref{ratiocond}. Thus, $p \in \mathscr{P}_2$ by Theorem \ref{thm:altchar}.
\end{proof}

We conclude by providing a novel characterization for polynomials that preserve all nonnegative circulant matrices or order two. If 
\begin{equation}
\label{gentwocirc}
A = 
\begin{bmatrix}
a & b \\
b & a
\end{bmatrix} 
=
\frac{1}{2}
\begin{bmatrix}
(a+b)+(a-b) & (a+b)-(a-b) \\
(a+b)-(a-b) & (a+b)+(a-b)
\end{bmatrix},
\end{equation}
then by \eqref{twocirc} and \eqref{polycirc}
\[ 
p(A) = \frac{1}{2} 
\begin{bmatrix}
p(\rho) + p(\mu)  & p(\rho) - p(\mu) \\
p(\rho) - p(\mu)  & p(\rho) + p(\mu)
\end{bmatrix} \ge 0,
\]
where $\rho = a + b$ and $\mu = a - b$. We immediately obtain the following result. 

\begin{theorem}
[cf.~{\cite[Theorem 10]{bh2008}}]
If $p \in \mathbb{R}[x]$, then $p$ preserves all two-by-two nonnegative circulant matrices of the form \eqref{gentwocirc} if and only if $p',p_e,p_o \in \mathscr{P}_1$. 
\end{theorem}

\section*{Acknowledgement}

We thank the anonymous referee for their careful review and thoughtful suggestions that greatly improved this work. 

\bibliography{nonpoly}
\bibliographystyle{abbrv}

\end{document}